\documentclass[12pt]{article}
\usepackage{latexsym,amsmath,amsthm,verbatim,ifthen,amssymb}
\usepackage{mathstuff}
\usepackage[all]{xy}

\newtheorem{theorem}{Theorem}
\newtheorem{corollary}[theorem]{Corollary}

\newtheorem{proposition}[theorem]{Proposition}

\newtheorem{definition}[theorem]{Definition}
\newtheorem{remark}[theorem]{Remark}
\date{}

\begin{document}
\title{The Ganea conjecture for rational approximations of sectional category}
\author{J.G. Carrasquel-Vera\let\thefootnote\relax\footnote{This work was supported by FEDER through the Ministerio de Educaci\'on y Ciencia project MTM2010-18089.}}

\maketitle
\begin{abstract}
We give bounds for the module sectional category of products of maps which generalise a theorem of Jessup for Lusternik-Schnirelmann category. We deduce also a proof of a Ganea type conjecture for topological complexity. This is a first step towards proving the Ganea conjecture for topological complexity in the rational context.
\end{abstract}

 \vspace{0.5cm}
 \vspace{0.2cm}

\maketitle

\section*{Introduction}
The sectional category of a map $f\colon X\rightarrow Y$, $\secat(f)$, defined by Schwarz in \cite{Schwarz66}, is the smallest $m$ such that there exist $m+1$ local homotopy sections for $f$ whose domains form an open cover of $Y$. One can point out two important cases of sectional category. The first one is the well known Lusternik-Schnirelmann (LS) category of a path-connected space $X$, which can be seen as the sectional category of the inclusion of the base point of $X$: \[\cat(X)=\secat(*\hookrightarrow X).\] More generally, if $f\colon X\rightarrow Y$ is a continuous map with homotopy fibre $i\colon F\rightarrow X$ then \[\cat(f)=\secat(i).\] The second one is the (higher) topological complexity, defined in \cite{Farber03} and generalised in \cite{Rudyak10}, \[\tc_n(X)=\secat(\Delta_n),\] where $\Delta_n\colon X\hookrightarrow X^n$ denotes the diagonal inclusion. It is known \cite{Farber03} that $\tc(X):=\tc_2(X)$ measures the motion planning complexity of a mechanical system for which $X$ is the configuration space.\\

Denote $S^k$ the $k$-dimensional sphere. Ganea conjectured in \cite{Ganea71} that $\cat(X\times S^k)=\cat(X)+\cat(S^k)$. This conjecture was found to be false by Iwase in \cite{Iwase98} but proven to be true for rational spaces. The latter was done in two steps. First Jessup proves the conjecture for a weaker invariant called module LS category (see below) \cite{Jessup90} \[\mcat(X\times S^k)=\mcat(X)+\mcat(S^k).\] Then Hess proved that, rationally, module LS category equals LS category \cite{He91} \[\cat(X)=\mcat(X).\]

The goal of this paper is to generalise Jessup's theorem to sectional category. In order to do so, we will use standard rational homotopy theory techniques. We will therefore always consider simply connected CW complexes with finite Betti numbers. In particular we denote by $X_0$ and $f_0$ the rationalisation of a space $X$ and a continuous map $f$. We denote $\apl(X)$ the commutative differential graded algebra (\cdga\ for short) of Sullivan's piecewise linear forms on $X$. The reader is referred to \cite{Bible} for the basis on rational homotopy theory.\\

Sectional category admits a nice presentation using the Ganea construction. Let $f\colon X\rightarrow Y$ be a continuous map. One can construct an \emph{$m$-Ganea map} for $f$, $G_m(f)\colon P^m(f)\rightarrow Y$, as $G_0(f)=f$ and \[G_m(f)=f*_Y\cdots *_Y f\colon X*_Y\cdots *_Y X\rightarrow Y,\] the iterated join $f$ with itself $m+1$ times \cite{Ma76}. When $Y$ admits a partition of the unity (for instance, when $Y$ is normal) the $m$-Ganea map can be used to \emph{glue together} $m+1$ local sections of $f$ so that $\secat(f)\le m$ if and only if $G_m(f)$ admits a homotopy section, \cite{Schwarz66}. Using $G_m(f)$ we obtain a characterisation of $\secat(f_0)$ and a definition of $\msecat(f)$. In fact, it follows directly from Sullivan's theory of minimal models \cite{Su77} that $\secat(f_0)\le m$ if and only if \[\apl(G_m(f))\colon \apl(Y)\rightarrow \apl(G_m(f))\] admits a homotopy retraction (see Section \ref{sec:ModSecat}) in the category of \cdga s \cite{Carrasquel10}.\\

\begin{definition}[\cite{FGKV06}]\ 
\begin{enumerate}
\item[(i)] The \emph{module sectional category} of a map $f$, $\msecat(f)$, is the smallest $m$ such that $\apl(G_m(f))$ admits a homotopy retraction in the category of $\apl(Y)$-modules \cite{FGKV06}. 
\item[(ii)] The module topological complexity of $X$ is $\mtc_n(X)=\msecat(\Delta_n)$.
\end{enumerate}
\end{definition}

In this paper we study the relations between $\msecat(f\times g)$, $\msecat(f)$ and $\msecat(g)$ being $g$ a map. In particular we prove

\begin{theorem}\label{th:main}
Suppose $f\colon X\rightarrow Y$ and $g\colon X'\rightarrow Y'$ are maps with $X_0$ a Poincar\'e duality complex. If $f_0$ and $g_0$ admit homotopy retractions then \[\msecat(f\times g)=\msecat(f)+\msecat(g).\]
\end{theorem}

We deduce:

\begin{corollary}
If $X_0$ is a Poincar\'e duality complex, then \[\mtc_n(X\times Y)=\mtc_n(X)+\mtc_n(Y).\]
\end{corollary}

\section{Module sectional category and products}\label{sec:ModSecat}
We will start with a brief recall of some content of \cite{Carrasquel10} and \cite{Carrasquel14c} that will be used later on. Throughout this paper we will work with commutative differential graded algebras over $\mQ$ whose differential increases the degree. Given a $\cdga$ $(A,d)$, an $(A,d)$-module is a graded differential $\mQ$-vector space $(M,d)$ with a degree $0$ action of $A$ verifying $d(ax)=d(a)x+(-1)^{\mbox{deg}(a)}ad(x)$. A \emph{homotopy retraction} of \cdga\ (resp. $(A,d)$-module) for a \cdga\ morphism  $\psi\colon (A,d)\rightarrow (B,d)$ is a \cdga\ (resp. $(A,d)$-module) morphism $r\colon (A\otimes\Lambda V,D)\rightarrow (A,d)$ such that $r\circ i=\id_A$ where $i\colon (A,d)\cofib (A\otimes\Lambda V,D)$ is a relative Sullivan model for $\psi$.

\begin{definition}
Let $\varphi\colon (A,d)\rightarrow (B,d)$ be a surjective morphism of \cdga s and consider the projection \[\rho_m\colon (A,d)\rightarrow \pa\frac{A}{(\ker\ \varphi)^{m+1}},\ol{d}\pb.\] Define:
\begin{itemize}
\item[(i)] $\asecat(\varphi)$ the smallest $m$ such that $\rho_m$ admits a homotopy retraction of \cdga.
\item[(ii)] $\amsecat(\varphi)$ the smallest $m$ such that $\rho_m$ admits a homotopy retraction of $(A,d)$-module.
\item[(iii)] $\ahsecat(\varphi)$ the smallest $m$ such that $\rho_m$ is homology injective.
\end{itemize}
\end{definition}

In order to give topological consequences to our algebraic results we will use the main theorem of \cite{Carrasquel14c}, which reads
\begin{theorem}
Let $f$ be a map modelled by the \cdga\ morphism $\varphi$. If $\varphi$ admits a section then:
\begin{itemize}
\item[i)] $\secat(f_0)=\asecat(\varphi)$.
\item[ii)] $\msecat(f)=\amsecat(\varphi)$.
\item[iii)] $\hsecat(f)=\ahsecat(\varphi)$.
\end{itemize}
\end{theorem}


We now prove the \emph{sub-additivity} of $\asecat$ type invariants for the tensor product of morphisms of \cdga s. 
\begin{proposition}\label{prop:SubAddSc}
Let $\varphi_i\colon (A_i,d)\longrightarrow (B_i,d)$, $i=1,2$, be surjective \cdga\ morphisms. Then
\begin{itemize}
\item[i)] $\asecat(\varphi_1\otimes\varphi_2)\le\asecat(\varphi_1)+\asecat(\varphi_2).$
\item[ii)] $\amsecat(\varphi_1\otimes\varphi_2)\le\amsecat(\varphi_1)+\amsecat(\varphi_2).$
\item[iii)] $\ahsecat(\varphi_1\otimes\varphi_2)\le\ahsecat(\varphi_1)+\ahsecat(\varphi_2).$
\end{itemize}
\end{proposition}

\begin{proof}
Denote $K_i=\ker \varphi_i$, $i=1,2$ and $L=\ker\ \varphi_1\otimes\varphi_2$. Since $L=K_1\otimes A_2+A_1\otimes K_2$, then for all $m,n\ge 1$  we have that $L^{m+n+1}\subset K_1^{m+1}\otimes A_2+A_1\otimes K_2^{n+1}$. This induces commutative diagram
\[\xymatrix{
&A_1\otimes A_2\ar[dl]_{\rho_3}\ar[dr]^{\rho_1\otimes\rho_2} &\\
\frac{A_1\otimes A_2}{L^{m+n+1}}\ar[rr]& &\frac{A_1}{K_1^{m+1}}\otimes \frac{A_2}{K_2^{n+1}},\\
}\]
which combined with the relative lifting lemma \cite[Prop. 12.9]{Bible} establishes the proposition. 
\end{proof}

\begin{corollary}
If $f$ and $g$ are continuous maps such that $f_0$ and $g_0$ admit homotopy retractions, then
\begin{itemize}
\item[i)] $\secat(f_0\times g_0)\le\secat(f_0)+\secat(g_0).$
\item[ii)] $\msecat(f\times g)\le\msecat(f)+\msecat(g).$
\item[iii)] $\hsecat(f\times g)\le\hsecat(f)+\hsecat(g).$
\end{itemize}
\end{corollary}

Now recall that there is a general procedure to compute $\msecat(f)$ from any surjective model $\varphi\colon (A,d)\rightarrow (B,d)$ even if it does not have a section. First we introduce the invariant $\msecat(\varphi)$ for a surjective morphism with kernel $K$.
\begin{definition}
The \emph{module sectional category} of $\varphi$ is the smallest integer $m$ such that exists a morphism of $(A^{\otimes m+1},d)$-modules \[r\colon\pa A^{\otimes m+1}\otimes\Lambda W,D\pb\rightarrow (A,d)\] making commutative the diagram
\[\xymatrix@C=2cm{
(A^{\otimes m+1},d)\ar@{>->}[r]^-j\ar[d]_{\mu}&\pa A^{\otimes m+1}\otimes\Lambda W,D\pb \ar[dl]^-r\\
(A,d), &\\
}\]
where $j$ is a relative Sullivan model for the projection $(A^{\otimes m+1},d)\rightarrow\pa \frac{A^{\otimes m+1}}{K^{\otimes m+1}},\ol{d}\pb$ and $\mu$ is the multiplication morphism.
\end{definition}
By \cite{Carrasquel10}, one has that $\msecat(f)=\msecat(\varphi)$, for any surjective model $\varphi$ of $f$. Recall that the nilpotency of an ideal $I$ is defined as the greatest integer $m$ such that $I^{m+1}\ne \ca 0\cb$. We can now prove

\begin{proposition}
Let $\varphi_i\colon (A_i,d)\rightarrow (B_i,d)$, $i=1,2$, be surjective \cdga\ morphisms. Then
\begin{itemize}
\item[i)] $\amsecat(\varphi_1\otimes\varphi_2)\ge \amsecat(\varphi_1)+\ahsecat(\varphi_2).$
\item[ii)] $\msecat(\varphi_1\otimes\varphi_2)\ge \msecat(\varphi_1)+\nil \ker\ H(\varphi_2).$
\end{itemize}
\end{proposition}

\begin{proof}
We begin proving i). Recall the notation from the proof of Proposition \ref{prop:SubAddSc} and suppose $\ahsecat(\varphi_2)=n$, then there must exist a cycle $\omega\in K_2^n$ representing a non-zero class of $H^*(A_2,d)$. One can therefore decompose $A_2=\mQ\cdot\omega\oplus M$ with $d(M)\subset M$ and \[A_1\otimes A_2=(A_1\otimes\mQ\cdot\omega)\oplus (A_1\otimes M).\] Now define the map $\alpha\colon A_1\rightarrow A_1\otimes A_2$ as $\alpha(a):=a\otimes\omega$. Observe that $\alpha$ is an $(A_1,d)$-module morphism of degree $|\omega|$. Define also the $(A_1,d)$-module morphism $\beta\colon A_1\otimes A_2\rightarrow A_1$ as $\beta(\omega):=1$ and $\beta(M)=0$. It is obvious that $\beta$ is a retraction for $\alpha$ and that $\alpha(K_1^{m+1})\subset L^{m+n+1}$. Then $\alpha$ induces a commutative $(A_1,d)$-module diagram
\[\xymatrix{
A_1\ar[dd]_{\rho_1}\ar@{>->}[dr]^{j_1}\ar[rrr]^-\alpha& & & A_1\otimes A_2\ar[dd]^{\rho_3}\ar@{>->}[dl]_{j_3}\\
&\bullet\ar[dl]^{\simeq}\ar@{-->}[r]_{\tilde{\alpha}}&\bullet\ar[dr]_{\simeq}&\\
\frac{A_1}{K^{m+1}}\ar[rrr]_-{\overline{\alpha}}&&&\frac{A_1\otimes A_2}{L^{m+n+1}},\\ 
}\]
where $j_1$ and $j_3$ are relative models for $\rho_1$ and $\rho_3$, respectively, and where $\tilde{\alpha}$ is induced by \cite[Prop. 6.4]{Bible}. If $j_3$ admits a retraction as $(A_1\otimes A_2)$-module, $r$, then $\beta\circ r\circ\tilde{\alpha}$ is a retraction for $j_1$.\\

Let us now prove ii). If $\nil \ker\ H(\varphi_2)=n$ then there exist cycles $\omega_1,\ldots,\omega_n\in \ker \varphi_2$ such that $[\omega_1\cdots \omega_n]\ne 0$ in $H(A_2)$. We can therefore define an $(A^{\otimes m+1}_1,d)$-module morphism $\gamma\colon A_1^{\otimes m+1}\rightarrow (A_1\otimes A_2)^{\otimes m+n+1}$ as $\alpha(x):=x\otimes \omega_1\otimes\cdots\otimes\omega_n\in A_1^{\otimes m+1}\otimes A_2^{\otimes n}$. Let also $\alpha$ and $\beta$ be the maps of previous case taking $\omega:=\omega_1\cdots\omega_n$. Since $\gamma(K_1^{\otimes m+1})\subset L^{\otimes m+n+1}$, we have a commutative $(A_1^{\otimes m+1},d)$-module diagram

\[\xymatrix@C=2cm{
\frac{A_1^{\otimes m+1}}{K_1^{\otimes m+1}}\ar[r]^-{\overline{\gamma}}&\frac{(A_1\otimes A_2)^{\otimes m+n+1}}{L^{\otimes m+n+1}}\\
A_1^{\otimes m+1}\ar[u]\ar[d]_\mu\ar[r]^-\gamma&(A_1\otimes A_2)^{\otimes m+n+1\ar[u]\ar[d]^\mu}\\
A_1\ar[r]_\alpha&A_1\otimes A_2.\\
}\]
and the result follows in a similar way as i).
\end{proof}

\begin{corollary}
Let $f$ and $g$ be continuous maps then \[\msecat(f)+\nil\ \ker\ H^*(g,\mQ)\le \msecat(f\times g).\] Moreover, if $f_0$ and $g_0$ admit homotopy retractions, then \[\msecat(f)+\hsecat(g)\le\msecat(f\times g)\le\msecat(f)+\msecat(g).\]
\end{corollary}

In \cite{Carrasquel14b} it was proven that if the base space of a map $g$ is a Poincar\'e duality complex, then $\msecat(g)=\hsecat(g)$. This implies Theorem \ref{th:main}.

\begin{remark}
Stanley in \cite{Stanley02} gives an example of two maps $f,g$ such that $\cat(f_0\times g_0)<\cat(f_0)+\cat(g_0)$. By taking homotopy cofibres we get examples of maps for which $\secat(f_0\times g_0)<\secat(f_0)+\secat(g_0)$.
\end{remark}

\section{Applications to topological complexity}
Recall from \cite{Carrasquel14b} that, if $X$ is a Poincar\'e duality complex, then $\htc_n(X)=\mtc_n(X)$, and $\htc_n(X):=\hsecat(\Delta_n)$. We can then deduce 
\begin{theorem}
Let $X, Y$ be spaces, then 
\[\mtc_n(X)+\htc_n(Y)\le \mtc(X\times Y)\le \mtc_n(X)+\mtc_n(Y).\] Moreover, if $Y$ is  Poincar\'e duality complex, then \[\mtc_n(X\times Y)=\mtc_n(X)+\mtc_n(Y).\]
\end{theorem}

In particular we extend \cite[Theorem 1.6]{Jessup12}, \[\mtc_n(X\times S^k)=\mtc_n(X)+\mtc_n(S^k).\]

\bibliography{bibliography.bib}{}
\bibliographystyle{plain}

\vspace{2cm}
\noindent Institut de Recherche en Math\'ematique et Physique,\\
Universit\'e catholique de Louvain,\\
2 Chemin du Cyclotron,\\
1348 Louvain-la-Neuve, Belgium.\\
E-mail: \texttt{jose.carrasquel@uclouvain.be}
\end{document}